\documentclass{amsart}
\usepackage{amssymb,amstext,amsmath,amscd,amsthm,amsfonts,enumerate,graphicx,latexsym,stmaryrd,multicol}
\usepackage[usenames]{color}
\usepackage{setspace}
\usepackage{bm}
\usepackage[all]{xy}

\newtheorem{thm}{Theorem}[section]
\newtheorem{lem}[thm]{Lemma}
\newtheorem{prop}[thm]{Proposition}
\newtheorem{cor}[thm]{Corollary}
\theoremstyle{definition}
\newtheorem{dfn}[thm]{Definition}

\newtheorem{ex}[thm]{Example}

\newtheorem{rmk}[thm]{Remark}

\theoremstyle{remark}

\newtheorem*{ac}{Acknowlegments}

\newtheorem*{proof of claim}{Proof of Claim}

\numberwithin{equation}{thm}
\def\ass{\operatorname{Ass}}
\def\ann{\operatorname{Ann}}
\def\cm{\mathsf{CM}}

\def\depth{\operatorname{depth}}
\def\Ext{\operatorname{Ext}}

\def\ge{\geqslant}
\def\height{\operatorname{ht}}

\def\ker{\operatorname{Ker}}
\def\le{\leqslant}

\def\m{\mathfrak{m}}

\def\p{\mathfrak{p}}

\def\q{\mathfrak{q}}

\def\spec{\operatorname{Spec}}
\def\supp{\operatorname{Supp}}

\def\V{\mathrm{V}}

\def\codepth{\operatorname{codepth}}
\def\grade{\operatorname{grade}}

\def\reg{\mathsf{Reg}}

\def\C{\mathsf{C}}

\def\Z{\mathbb{Z}}
\begin{document}
\allowdisplaybreaks
\title{Asymptotic stability of depths of localizations of modules}
\author{Kaito Kimura}
\address{Graduate School of Mathematics, Nagoya University, Furocho, Chikusaku, Nagoya 464-8602, Japan}
\email{m21018b@math.nagoya-u.ac.jp}
\begin{abstract}
Let $R$ be a commutative noetherian ring, $I$ an ideal of $R$, and $M$ a finitely generated $R$-module.
The asymptotic behavior of the quotient modules $M/I^n M$ of $M$ is an actively studied subject in commutative algebra.
The main result of this paper asserts that the depth of the localization of $M/I^n M$ at any prime ideal of $R$ is stable for large integers $n$ that do not depend on the prime ideal, if the module $M$ or $M/I^n M$ is Cohen--Macaulay for some $n>0$, or the ring $R$ is one of the following: 
a homomorphic image of a Cohen--Macaulay ring, a semi-local ring, an excellent ring, a quasi-excellent and catenary ring, and an acceptable ring.
\end{abstract}
\maketitle
\section{Introduction}

Throughout the present paper, all rings are assumed to be commutative and noetherian.
Let $R$ be a ring, $I$ an ideal of $R$, and $M$ a finitely generated $R$-module.
The asymptotic behavior of the quotient modules $M/I^n M$ of $M$ for large integers $n$ is one of the most classical subjects in commutative algebra. 
Among other things, the asymptotic stability of the associated prime ideals and depths of $M/I^n M$ has been actively studied.
Brodmann \cite{B} proved that the set of associated prime ideals of $M/I^n M$ is stable for large $n$.
Kodiyalam \cite{Ko} showed that the depth of $M/I^n M$ attains a stable constant value for all large $n$ when $R$ is local.
There are a lot of studies about this subject; see \cite{Br, M, RS} for instance.

The purpose of this paper is to proceed with the study of the above subject.
In particular, we consider the existence of an integer $k$ such that $\depth (M/I^t M)_\p=\depth (M/I^k M)_\p$ for all integers $t\ge k$ and all prime ideals $\p$ of $R$.
In this direction, 
by using the openness of the codepth loci of modules over excellent rings studied by Grothendieck \cite{G},
Rotthaus and \c{S}ega \cite{RS} proved that such an integer $k$ exists if $R$ is excellent, $M$ is Cohen--Macaulay, and $I$ contains an $M$-regular element.
We aim to improve their theorem by applying the ideas of their proofs.
However, in our proof, we use the methods developed in \cite{K} not those of Grothendieck.

The main result of this paper is the following theorem; for the definition of an acceptable ring in the sense of Sharp \cite{S} see Definition \ref{def rings}.
Obviously, we may replace all $\bar{R}$ in the result below with $R$.
It gives a common generalization of the above mentioned theorems proved in \cite{Ko} and \cite{RS}.

\begin{thm}[Corollary \ref{main cor}]\label{main}
Let $R$ be a ring, $I$ an ideal of $R$, and $M$ a finitely generated $R$-module.
Put $\bar{R}=R/(I+\operatorname{Ann}_R (M))$.
Then there is an integer $k>0$ such that
$$
\depth (M/I^t M)_\p=\depth (M/I^k M)_\p
$$
for all integers $t\ge k$ and all prime ideals $\p$ of $R$ in each of the following cases.
\begin{enumerate}[\rm(1)]
\item $M$ is Cohen--Macaulay. 
\item $M/I^n M$ is Cohen--Macaulay for some $n>0$.
\item $\bar{R}$ is a homomorphic image of a Cohen--Macaulay ring. 
\item $\bar{R}$ is semi-local.
\item $\bar{R}$ is excellent.
\item $\bar{R}$ is quasi-excellent and catenary.
\item $\bar{R}$ is acceptable.
\end{enumerate}
\end{thm}

The organization of this paper is as follows. 
In Section 2, we give several definitions and basic lemmas about graded rings and modules.
In Section 3, we study the openness of the codepth loci of graded modules.
We give a sufficient condition for the codepth loci of a graded module to be open, and for the depths of localizations of homogeneous components of a graded module to be eventually stable.
In Section 4, we prove Theorem \ref{main} and consider some examples.

\section{Definitions and lemmas}

This section is devoted to preliminaries for the later sections.
We give several definitions and basic lemmas about graded rings and modules.

In this section, we assume that $A=\bigoplus_{i\ge 0}A_i$ is a graded ring and that $M=\bigoplus_{i\in \Z}M_i$ is a finitely generated graded $A$-module.
The ring $A$ is a finitely generated $A_0$-algebra.
For any $i\in \Z$, the $A_0$-module $M_i$ is finitely generated.
Let $S$ be a multiplicatively closed subset of $A_0$.
Then $A_S=\bigoplus_{i\ge 0}(A_i)_S$ is also a graded ring, and $M_S=\bigoplus_{i\in \Z}(M_i)_S$ is a finitely generated graded $A_S$-module.
In particular, $A_\p$ is a graded ring having the local base ring $(A_0)_\p$ for any prime ideal $\p$ of $A_0$.
Similarly, $A/IA$ and $M/IM$ are graded for any ideal $I$ of $A_0$.
A graded ring $A$ which as an $A_0$-algebra is generated by elements of $A_1$ will be called \textit{homogeneous}.
Every ring $R$ is a graded ring $A$ with $A_0=R$ and $A_i=0$ for all $i\ge0$. 

We denote by $\operatorname{Ann}_{A_0} (M)$ the annihilator ideal of $M$.
The \textit{dimension} of $M$ as an $A_0$-module is given by $\dim_{A_0}(M) =\dim (A_0/\operatorname{Ann}_{A_0} (M))$.
Let $A_0$ be a local ring.
In general, $M$ is not finitely generated as an $A_0$-module.
Here, the \textit{depth} of $M$ as an $A_0$-module is defined as follows; see \cite[Definition 1.2.1]{RS}.
Note that this coincides with the one defined in \cite[Definition 9.1.1]{BH}.

\begin{dfn}
Let $(A_0, \m_0)$ be a local ring.
If $M$ is the zero module, then we set $\depth_{A_0}(M)=\infty$;
otherwise, we define $\depth_{A_0}(M)={\rm sup}\{n \ge 0 \mid$ there is an $M$-regular sequence $\bm{x}=x_1,\ldots,x_n$ in $\m_0 \}$.
Also, if $M$ is the zero module, then we set $\codepth_{A_0}(M)=-\infty$;
otherwise we define $\codepth_{A_0}(M)=\dim_{A_0}(M)-\depth_{A_0}(M)$.
\end{dfn}

In this paper, the following notation is used.

\begin{dfn}
Let $R$ be a ring, $I$ an ideal of $R$, and $n\ge 0$ an integer.
\begin{itemize}
\item $\V_{R}(I) = \{\p\in\spec (R)\mid I\subseteq\p\}$.
\item $\cm(R) = \{\p\in\spec (R)\mid \dim(R_\p) \le \depth(R_\p) \}$.
\item $\C_n^{A_0}(M) = \{\p\in\spec (A_0)\mid \codepth_{(A_0)_\p} (M_\p) \le n  \}$.
\item $\cm_{A_0} (M) = \{\p\in\spec (A_0)\mid \dim_{(A_0)_\p} (M_\p) \le \depth_{(A_0)_\p} (M_\p)  \} = \C_0^{A_0}(M)$.
\end{itemize}
\end{dfn}

We prepare several basic lemmas.
Some of the results below are proved in \cite[Section 1]{RS}.
However, due to some differences, we include proofs of those for the benefit of the reader.

\begin{lem}\label{RS1.1.1}
Suppose that $A$ is homogeneous.
Then there exists an integer $k$ such that $\ann_{A_0}(M_t)=\ann_{A_0}(M_k)$ for all integers $t\ge k$.
\end{lem}

\begin{proof}
There is an integer $l$ such that $M_{t+1}=A_1 M_t$ for any integer $t\ge l$ since $A$ is homogeneous and $M$ is a finitely generated $A$-module.
For any $t\ge l$, the ideal $\ann_{A_0}(M_{t+1})$ contains $\ann_{A_0}(M_t)$.
As $A_0$ is noetherian, there exists an integer $k$ such that $\ann_{A_0}(M_t)=\ann_{A_0}(M_k)$ for all integers $t\ge k$.
\end{proof}

\begin{lem}\label{RS1.3}
The function $F:\ass_A(M) \to \ass_{A_0}(M)$ defined by $F(P)=P\cap A_0$ is well defined and surjective.
In particular, $\ass_{A_0}(M)=\bigcup_{i\in \Z} \ass_{A_0}(M_i)$ is a finite set.
\end{lem}

\begin{proof}
For any prime ideal $P$ of $A$, there is the natural injection from $A_0/P\cap A_0$ to $A/P$.
Hence $F$ is well defined.
Let $\p\in\ass_{A_0}(M)$.
It follows from \cite[Theorem 6.2]{Mat} and \cite[Proposition 1.2.1]{BH} that there exists a prime ideal $Q$ of $A_\p$ which belongs to $\ass_{A_\p}(M_\p)$ and contains $\p A_\p$.
Then $Q=P A_\p$ for some $P\in \ass_A(M)$.
The ideal $Q$ is contained in the maximal ideal $\m=\p(A_0)_\p\oplus\bigoplus_{i>0} (A_i)_\p$ of $A_\p$ as $Q$ is graded; see \cite[Lemma 1.5.6]{BH}.
We easily see that $\p=P\cap A_0$.
\end{proof}

\begin{lem}\label{local lemma}
Let $(A_0, \m_0)$ be a local ring.
\begin{enumerate}[\rm(1)]
\item 
There is an equality $\dim_{A_0}(M)={\rm sup}\{\dim_{A_0}(M_i) \mid i\in \Z \}$.
\item 
One has the equality $\depth_{A_0}(M)={\rm inf}\{\depth_{A_0}(M_i) \mid i\in \Z \}$.
\item 
Let $0\to N\to M\to L\to 0$ be an exact sequence of finitely generated graded $A$-modules. Then 
$$
\depth_{A_0}(M)\ge{\rm min}\{\depth_{A_0}(N),\ \depth_{A_0}(L)\}.
$$
\item 
Suppose that a sequence $\bm{x}=x_1,\ldots,x_n$ of elements in $\m_0$ is an $M$-regular sequence. Then we have
$$
\dim_{A_0}(M)=\dim_{A_0}(M/\bm{x} M)+n, \ {\rm and} \ \depth_{A_0}(M)=\depth_{A_0}(M/\bm{x} M)+n.
$$
\end{enumerate}
\end{lem}

\begin{proof}
(1): Since $\ann_{A_0}(M)$ is contained in $\ann_{A_0}(M_i)$, we get $\dim_{A_0}(M)\ge \dim_{A_0}(M_i)$ for any $i\in \Z$, which means $\dim_{A_0}(M)\ge{\rm sup}\{\dim_{A_0}(M_i) \mid i\in \Z \}$.
Conversely, if $\p$ is a prime ideal of $A_0$ containing $\ann_{A_0}(M)$, then $M_\p$ is not the zero module because $M$ is finitely generated as an $A$-module.
So $(M_i)_\p\ne 0$ for some $i\in \Z$, and thus $\dim(R/\p)\le\dim_{A_0}(M_i)$.
This says that the other inequality holds.

(2): By definition, we observe that $\depth_{A_0}(M)\le{\rm inf}\{\depth_{A_0}(M_i) \mid i\in \Z \}$.
Let $d=\depth_{A_0}(M)$ and let $\bm{y}=y_1,\ldots,y_d$ be a $M$-regular sequence in $\m_0$.
The ideal $\m_0$ consists of zero-divisors of $M/\bm{y} M$.
It follows from \cite[Theorem 6.1]{Mat} and Lemma \ref{RS1.3} that $\m_0$ is in $\ass_{A_0}(M/\bm{y} M)=\bigcup_{i\in \Z} \ass_{A_0}(M_i/\bm{y} M_i)$.
The other inequality holds as $\bm{y}$ is a maximal $M_i$-regular sequence for some $i\in \Z$.

(3): We can choose an integer $i\in \Z$ such that $\depth_{A_0}(M)=\depth_{A_0}(M_i)$ by (2).
There is an exact sequence $0\to N_i\to M_i\to L_i\to 0$ of $A_0$-modules.
It follows from (2) and \cite[Proposition 1.2.9]{BH} that
$$
\depth_{A_0}(M)=\depth_{A_0}(M_i)\ge{\rm min}\{\depth_{A_0}(N_i),\ \depth_{A_0}(L_i)\} 
\ge{\rm min}\{\depth_{A_0}(N),\ \depth_{A_0}(L)\}.
$$

(4): The assertion follows from (1) and (2).
\end{proof}

\begin{lem}\label{graded lemma}
Let $\p$ be a prime ideal of $A_0$, and let $I=\ann_{A_0}(M)$.
\begin{enumerate}[\rm(1)]
\item 
If $\p$ belongs to $\supp_{A_0}(M)=\V_{A_0}(I)$, then $\dim_{(A_0)_\p}(M_\p)=\height (\p/I)$.
\item 
Suppose that a sequence $\bm{x}=x_1,\ldots,x_n$ of elements in $\p$ is an $M_\p$-regular sequence.
Then there exists $f\in A_0\setminus\p$ such that $\bm{x}$ is an $M_f$-regular sequence.
\item 
The prime ideal $\p$ belongs to $\ass_{A_0}(M)$ if and only if $\depth_{(A_0)_\p}(M_\p)=0$.
\end{enumerate}
\end{lem}

\begin{proof}
(1): It is seen that $\ann_{(A_0)_\p}(M_\p)=I(A_0)_\p$ since $M$ is a finitely generated $A$-module.

(2):  We may assume $n=1$.
Let $\varphi$ be the multiplication map of $M$ by $x_1$.
The submodule $\ker \varphi$ of $M$ is a finitely generated $A$-module, and $(\ker \varphi)_\p$ is the zero module.
We have $(\ker \varphi)_f=0$ for some $f\in A_0\setminus\p$, which means that $x_1$ is an $M_f$-regular element.
 
(3): It follows from \cite[Theorem 6.2]{Mat} that $\p$ is in $\ass_{A_0}(M)$ if and only if $\p(A_0)_\p$ is in $\ass_{(A_0)_\p}(M_\p)$.
Hence the ``only if'' part is trivial.
In order to prove the ``if'' part, suppose $\depth_{(A_0)_\p}(M_\p)=0$.
There is an integer $i\in \Z$ such that $\depth_{(A_0)_\p}(M_i)_\p=\depth_{(A_0)_\p}(M_\p)=0$ by Lemma \ref{local lemma} (2).
The prime ideal $\p$ belongs to the subset $\ass_{A_0}(M_i)$ of $\ass_{A_0}(M)$.
\end{proof}

\begin{rmk}\label{supp}
The subset $\supp_{A_0}(M)=\V_{A_0}(\operatorname{Ann}_{A_0} (M))$ of $\spec(A_0)$ is closed.
Therefore, if a prime ideal $\p$ of $A_0$ is not in $\supp_{A_0}(M)$, then the codepth locus $\C_n^{A_0}(M)$ contains a nonempty open subset $(\spec(A_0)\setminus\supp_{A_0}(M))\cap\V_{A_0}(\p)$ of $\V_{A_0}(\p)$ for any $n\ge 0$.
\end{rmk}

We close this section by stating an elementary lemma about open subsets of the spectrum of rings.

\begin{lem}\label{RS4.0}
Let $R$ be a ring, and let $\{U_n^t\}_{n\ge0 ,t\in\Z}$ be a family of open subsets of $\spec(R)$.
Suppose that $U_n^t$ is contained in both $U_n^{t+1}$ and $U_{n+1}^t$ for all $t\in\Z$ and all $n\ge0$.
Then there is an integer $k$ such that $U_n^t=U_n^k$ for all $t\ge k$ and all $n\ge 0$.
\end{lem}

\begin{proof}
There is an integer $k_1$ such that $U_t^t=U_{k_1}^{k_1}$ for all $t\ge k_1$ since $R$ is noetherian.
Also, there is an integer $k_2$ such that $U_n^t=U_n^{k_2}$ for all $t\ge k_2$ and all $0\le n< k_1$.
Put $k={\rm max}\{k_1,k_2\}$.
For all $t\ge k$ and all $n\ge 0$, we have $U_n^t=U_n^k$ by considering the case $0\le n< k_1$ and the case $k_1\le n$ separately.
\end{proof}

\section{The openness of the codepth loci of graded modules}

In this section, we study the openness of the codepth loci of graded modules.
The purpose of this section is to give a sufficient condition for the depths of localizations of homogeneous components of a graded module to be eventually stable.
As in the previous section, we assume in this section that $A=\bigoplus_{i\ge 0}A_i$ is a graded ring and that $M=\bigoplus_{i\in \Z}M_i$ is a finitely generated graded $A$-module.

Below is a well known result of Grothendieck \cite[(6.11.5)]{G}.

\begin{lem}\label{f.g. gene closed}
Let $(R,\m)$ be a local ring, $\p$ a prime ideal of $R$, and $N$ a finitely generated $R$-module.
Then we have 
$$
\codepth_{R_\p}(N_\p) \le \codepth_{R}(N).
$$
\end{lem}

\begin{proof}
We may assume that $\p$ belongs to $\supp_R(N)$.
Let $\bm{x}=x_1,\ldots,x_n$ be a maximal $N$-regular sequence in $\p$.
There exists an associated prime ideal $\q$ of $N/\bm{x}N$ containing $\p$.
By \cite[Theorem 17.2]{Mat}, we obtain
\begin{align*}
\depth_R(N)-\depth_{R_\p}(N_\p)&\le \depth_R(N)-n= \depth_R(N/\bm{x}N)\le \dim(R/\q) \\
&\le \dim(R/\p)\le \dim_R(N)-\dim_{R_\p}(N_\p).
\end{align*}
This says that the assertion holds.
\end{proof}

Lemma \ref{f.g. gene closed} can be extended as follows.
The proof of \cite[Lemma 2.5]{RS}, which states a similar fact to below follows the ideas of Grothendieck's proof given in \cite{G}.
Our proof is simpler than that.

\begin{lem}\label{gene closed}
Let $\p$ and $\q$ be prime ideals of $A_0$ with $\p\subseteq\q$.
Then we have 
$$
\codepth_{(A_0)_\p}(M_\p) \le \codepth_{(A_0)_\q}(M_\q).
$$
\end{lem}

\begin{proof}
By (1) and (2) of Lemma \ref{local lemma}, we can take integers $i, j, k, l \in\Z$ such that 
\begin{align*}
&\dim_{(A_0)_\p}(M_\p)=\dim_{(A_0)_\p}(M_i)_\p, \ 
\depth_{(A_0)_\p}(M_\p)=\depth_{(A_0)_\p}(M_j)_\p, \\
&\dim_{(A_0)_\q}(M_\q)=\dim_{(A_0)_\q}(M_k)_\q,\ {\rm and} \ 
\depth_{(A_0)_\q}(M_\q)=\depth_{(A_0)_\q}(M_l)_\q.
\end{align*}
It follows from Lemma \ref{f.g. gene closed} that
$$
\codepth_{(A_0)_\p}(M_\p) = \codepth_{(A_0)_\q}(N_\p) \le 
\codepth_{(A_0)_\p}(N_\q) = \codepth_{(A_0)_\q}(M_\q)
$$
as $N:=M_i\oplus M_j\oplus M_k\oplus M_l$ is a finitely generated $A_0$-module.
\end{proof}

We consider the openness of the codepth loci of a graded module to state the main result of this paper.
The following theorem is a graded version of \cite[Theorem 5.4]{K}.

\begin{thm}\label{cm2}
Let $\p\in\cm_{A_0}(M)$.
If $\cm(A_0/\p)$ contains a nonempty open subset of $\spec (A_0/\p)$, then $\cm_{A_0}(M)$ contains a nonempty open subset of $\V_{A_0}(\p)$.
\end{thm}

\begin{proof}
First of all, we may assume that $\p$ belongs to $\supp_{A_0}(M)$ by Remark \ref{supp}.
Also, note that we may assume $\supp_{A_0}(M)=\spec(A_0)$ by replacing $A$ with $A/\operatorname{Ann}_{A} (M)$, and can freely replace our ring $A$ with its localization $A_f$ for any element $f\in A_0\setminus\p$ to prove the theorem; see \cite[Lemmas 2.5 and  2.6]{K}.

Put $d=\depth_{(A_0)_\p}(M_\p)=\dim_{(A_0)_\p}(M_\p)=\height\p$.
We can choose a sequence $\bm{x}=x_1,\ldots,x_d$ in $\p$ such that it is an $M_\p$-regular sequence and $\height\bm{x}A_0=d$ by Lemmas  \ref{RS1.3}, \ref{local lemma} (4), and \ref{graded lemma} (3).
We may assume that $A_0/\p$ is Cohen--Macaulay and that $\p^r$ is contained in $\bm{x}A_0$ for some $r>0$.
Also, Lemma \ref{graded lemma} (2) yields that we may assume that $\bm{x}$ is an $M$-regular sequence.
Set $\overline{A}=A/\bm{x} A$, $\overline{\p}=\p\overline{A}$, and $\overline{M}=M/\bm{x} M$.
Thanks to \cite[Theorem 24.1]{Mat}, for each $0\le i\le r-1$, we may assume that the graded $A$-module $\overline{\p}^{i}\overline{M}/\overline{\p}^{i+1}\overline{M}$ is free as an $A_0/\p$-module.

We claim that $\cm_{A_0}(M)$ contains $\V_{A_0}(\p)$.
Let $\q$ be a prime ideal of $A_0$ containing $\p$.
We obtain 
$$
\depth_{(A_0)_\q} (\overline{\p}^{i}\overline{M}/\overline{\p}^{i+1}\overline{M})_\q=\depth (A_0/\p)_\q=\height(\q/\p)=\height(\q/\bm{x} A_0)
$$
for each $0\le i\le r-1$.
It follows from (3) and (4) of Lemma \ref{local lemma} that
$$
\depth_{(A_0)_\q}(M_\q)=\depth_{(A_0)_\q}(\overline{M}_\q)+d \ge \height(\q/\bm{x} A_0)+d=\height\q.
$$
Hence, $\q$ belongs to  $\cm_{A_0}(M)$.
\end{proof}

Below is a direct corollary of Theorem \ref{cm2}.
It is a graded version of \cite[Corollary 5.5 (1)]{K}.

\begin{cor}\label{cm openness}
Suppose that $\cm(A_0/\p)$ contains a nonempty open subset of $\spec (A_0/\p)$ for any $\p\in\supp_{A_0}(M)\cap\cm_{A_0}(M)$.
Then $\cm_{A_0}(M)$ is an open subset of $\spec(A_0)$.
\end{cor}

\begin{proof}
The assertions follow from \cite[Theorem 24.2]{Mat}, Remark \ref{supp}, Lemma \ref{gene closed} and Theorem \ref{cm2}.
\end{proof}

When the base ring of $A/\operatorname{Ann}_{A} (M)$ is catenary, Theorem \ref{cm2} can be extended as follows.

\begin{thm}\label{codepth}
Let $n\ge 0$ be an integer and let $\p\in\C_n^{A_0}(M)$.
Suppose that that the ring $A_0/\operatorname{Ann}_{A_0} (M)$ is catenary.
If $\cm(A_0/\p)$ contains a nonempty open subset of $\spec (A_0/\p)$, then $\C_n^{A_0}(M)$ contains a nonempty open subset of $\V_{A_0}(\p)$.
\end{thm}

\begin{proof}
In an analogous way as at the beginning of the proof of Theorem \ref{cm2}, we may assume $\p\in\supp_{A_0}(M)$ and can freely replace our ring $A$ with its localization $A_f$ for any element $f\in A_0\setminus\p$ to prove the theorem.

We prove the theorem by induction on $n$.
We have already shown the case where $n=0$ in Theorem \ref{cm2}.
Let $n>0$ and $d=\depth_{(A_0)_\p}(M_\p)$.
By the induction hypothesis, we may assume $\codepth_{(A_0)_\p}(M_\p)=n$.
Thanks to Lemma \ref{graded lemma} (2), we may assume that the following conditions are satisfied.
\begin{enumerate}[\rm (a)]
\item The prime ideal $\p$ contains any minimal prime ideal of $\operatorname{Ann}_{A_0} (M)$.
\item There is an $M$-regular sequence $\bm{x}=x_1,\ldots,x_d$ in $\p$.
\end{enumerate}
Set $N=M/\bm{x} M$. 
Note that the $\p$-torsion submodule $\Gamma_{\p}(N)$ of $N$ is finitely generated and graded as an $A$-module.
We easily see that $\supp_{A_0}(\Gamma_{\p}(N))=\V_{A_0}(\p)$; see Lemmas \ref{local lemma} (4) and \ref{graded lemma} (3).
Since 
$$
\dim_{(A_0)_\p}(\Gamma_{\p}(N))_\p=0<n=\codepth_{(A_0)_\p}(M_\p)=\codepth_{(A_0)_\p}(N_\p)=\dim_{(A_0)_\p}(N_\p),
$$
it is seen that $\p$ is in $\C_0^{A_0}(\Gamma_{\p}(N))$ and $\dim_{(A_0)_\p}(N/\Gamma_{\p}(N))_\p=\dim_{(A_0)_\p}(N_\p)=n$.
On the other hand, Lemma \ref{graded lemma} (3) implies $\depth_{(A_0)_\p}(N/\Gamma_{\p}(N))_\p>0$.
Thus $\p$ belongs to $\C_{n-1}^{A_0}(N/\Gamma_{\p}(N))$.
By the induction hypothesis, we may assume that the following condition (c) is satisfied.
\begin{enumerate}[\rm (c)]
\item The set $\V_{A_0}(\p)$ is contained in both $\C_0^{A_0}(\Gamma_{\p}(N))$ and $\C_{n-1}^{A_0}(N/\Gamma_{\p}(N))$.
\end{enumerate}

We prove that $\V_{A_0}(\p)$ is contained in $\C_n^{A_0}(M)$.
Let $\q$ be a prime ideal of $A_0$ containing $\p$.
Now the ring $A_0/\operatorname{Ann}_{A_0} (M)$ is catenary.
By (a) and (c), we have
\begin{align*}
\depth_{(A_0)_\q}(\Gamma_{\p}(N))_\q=\dim_{(A_0)_\q}(\Gamma_{\p}(N))_\q=\height(\q/\p)
&=\height(\q/\operatorname{Ann}_{A_0} (M))-\height(\p/\operatorname{Ann}_{A_0} (M))\\
&=\dim_{(A_0)_\q}(M_\q)-\dim_{(A_0)_\p}(M_\p) \\
&=\dim_{(A_0)_\q}(M_\q)-(n+d).
\end{align*}
Note that $\supp_{A_0}(\Gamma_{\p}(N))=\V_{A_0}(\p)$ is contained in $\supp_{A_0}(N/\Gamma_{\p}(N))$.
By (b) and (c), we get 
\begin{align*}
\depth_{(A_0)_\q}(N/\Gamma_{\p}(N))_\q \ge \dim_{(A_0)_\q}(N/\Gamma_{\p}(N))_\q-(n-1)
&=\dim_{(A_0)_\q}(N_\q)-(n-1) \\
&=(\dim_{(A_0)_\q}(M_\q)-d)-(n-1) \\
&>\dim_{(A_0)_\q}(M_\q)-(n+d).
\end{align*}
Therefore, we observe that 
\begin{align*}
\depth_{(A_0)_\q}(M_\q)=\depth_{(A_0)_\q}(N_\q)+d
&\ge {\rm min} \{\depth_{(A_0)_\q}(\Gamma_{\p}(N))_\q,\ \depth_{(A_0)_\q}(N/\Gamma_{\p}(N))_\q\}+d \\
&= \dim_{(A_0)_\q}(M_\q)-n
\end{align*}
by Lemma \ref{local lemma} (3), which means that $\q$ belongs to $\C_n^{A_0}(M)$.
\end{proof}

The same result as Corollary \ref{cm openness} holds for codepth loci.

\begin{cor}\label{codepth2}
Suppose that the ring $A_0/\operatorname{Ann}_{A_0} (M)$ is catenary.
\begin{enumerate}[\rm(1)]
\item 
Let $n\ge 0$ be an integer.
Suppose that $\cm(A_0/\p)$ contains a nonempty open subset of $\spec (A_0/\p)$ for any $\p\in\supp_{A_0}(M)\cap\C_n^{A_0}(M)$.
Then $\C_n^{A_0}(M)$ is open.
\item 
Suppose that $\cm(A_0/\p)$ contains a nonempty open subset of $\spec (A_0/\p)$ for any $\p\in\supp_{A_0}(M)$.
Then $\C_n^{A_0}(M)$ is open for any integer $n\ge 0$.
\end{enumerate}
\end{cor}

\begin{proof}
The assertions follow from \cite[Theorem 24.2]{Mat}, Remark \ref{supp}, Lemma \ref{gene closed} and Theorem \ref{codepth}.
\end{proof}

We study the asymptotic behavior of the depths of localizations of homogeneous components of a graded module.
We prepare the following basic lemma to state Lemma \ref{RS4.2}.

\begin{lem}\label{RS4.1}
Suppose that $A$ is homogeneous and that $(A_0, \m_0, k_0)$ is local.
Then there exists an integer $k$ such that $\depth_{A_0} (M_t)=\depth_{A_0} (M_k)$ for all integers $t\ge k$.
\end{lem}

\begin{proof}
It is seen that $\Ext_{A_0}^i(k_0,M)\simeq\bigoplus_{t\in\Z}\Ext_{A_0}^i(k_0,M_t)$ is a finitely generated graded $A$-module for any $0\le i\le \dim(A_0)$.
Since $\depth_{A_0} (M_t)={\rm inf}\{ i\mid \Ext_{A_0}^i(k_0,M_t)\ne 0 \}$ for each $t\in\Z$, the assertion follows from Lemma \ref{RS1.1.1}.
\end{proof}

Applying the ideas of the proof of \cite[Theorem 4.2]{RS}, we can prove the result below, which extends it.

\begin{lem}\label{RS4.2}
Suppose that $A$ is homogeneous.
Denote by $N_t$ the graded $A$-module $\bigoplus_{i\ge t}M_i$ for each $t\in\Z$.
If $\C_n^{A_0}(N_t)$ is open for all $t\in \Z$ and all $n\ge 0$, then there is an integer $k$ such that
$$
\depth_{(A_0)_\p}(M_t)_\p=\depth_{(A_0)_\p}(M_k)_\p
$$
for all integers $t\ge k$ and all prime ideals $\p$ of $A_0$.
\end{lem}

\begin{proof}
It follows from (1) and (2) of Lemma \ref{local lemma} that $\C_n^{A_0}(N_t)$ is contained in both $\C_n^{A_0}(N_{t+1})$ and $\C_{n+1}^{A_0}(N_t)$ for all $t\in\Z$ and all $n\ge 0$.
By Lemmas \ref{RS1.1.1} and \ref{RS4.0}, we can choose an integer $l\in\Z$ such that
\begin{equation}\label{4.2RS1}
J:=\ann_{A_0}(M_t)=\ann_{A_0}(M_l)\ \ {\rm and} \ \ 
U_n:=\C_n^{A_0}(N_t)=\C_n^{A_0}(N_l)
\end{equation}
for all $t\ge l$ and $n\ge 0$.
Any prime ideal of $A_0$ belongs to $U_n$ for some $n\ge 0$.
An analogous argument to the proof of Lemma \ref{RS4.0} shows that there exists an integer $m\ge 0$ such that $U_m=\spec(A_0)$.
For each $0\le n\le m-1$, we can write $\V_{A_0}(I_n)=\spec(A_0)\setminus U_n$ for some ideal $I_n$ of $A_0$.
The subset $\bigcup_{n=0}^{m-1} \ass_{A_0}(A_0/I_n)$ of $\spec(A_0)$ is finite.
It follows from Lemma \ref{RS4.1} that we can take $k\ge l$ such that
\begin{equation}\label{4.2RS2}
\depth_{(A_0)_\q}(M_t)_\q=\depth_{(A_0)_\q}(M_k)_\q
\end{equation}
for any $t\ge k$, and any $\q\in\bigcup_{n=0}^{m-1} \ass_{A_0}(A_0/I_n)$.

Let $\p$ be a prime ideal of $A_0$.
We claim that $\depth_{(A_0)_\p}(M_t)_\p=\depth_{(A_0)_\p}(M_k)_\p$ for all $t\ge k$.
We may assume that $\p$ contains $J$.
If $\p$ belongs to $U_0$, then we have
$$
\depth_{(A_0)_\p}(M_t)_\p\le \dim_{(A_0)_\p}(M_t)_\p\le \dim_{(A_0)_\p}(N_k)_\p\le
\depth_{(A_0)_\p}(N_k)_\p\le \depth_{(A_0)_\p}(M_t)_\p
$$
for all $t\ge k$ by (1) and (2) of Lemma \ref{local lemma}.
This means that the claim holds.
If $\p$ does not belong to $U_0$, then $\codepth_{(A_0)_\p}(N_k)_\p=n+1$ for some $0\le n\le m-1$.
As $\p$ is not in $U_n$, we see that $\q\subseteq\p$ for some $\q\in\ass_{A_0}(A_0/I_n)$.
By (1) and (2) of Lemma \ref{local lemma}, Lemma \ref{f.g. gene closed}, (\ref{4.2RS1}) and (\ref{4.2RS2}), it is seen that
$$
n+1\le \codepth_{(A_0)_\q}(N_k)_\q=\codepth_{(A_0)_\q}(M_t)_\q\le 
\codepth_{(A_0)_\p}(M_t)_\p\le \codepth_{(A_0)_\p}(N_k)_\p=n+1.
$$
for all $t\ge k$.
Hence we get $\codepth_{(A_0)_\p}(M_t)_\p=n+1$ for all $t\ge k$.
The claim follows from (\ref{4.2RS1}).
\end{proof}

\section{Asymptotic stability of depths of localizations of modules}

In this section, we prove the main result of this paper.
All of the results of Theorem \ref{main} are given as corollaries of the theorem below.

\begin{thm}\label{local depth}
Let $R$ be a ring, $I$ an ideal of $R$, and $M$ a finitely generated $R$-module.
Suppose that that the ring $\bar{R}:=R/(I+\operatorname{Ann}_R (M))$ is catenary and that $\cm(\bar{R}/\bar{\p})$ contains a nonempty open subset of $\spec (\bar{R}/\bar{\p})$ for any prime ideal $\bar{\p}$ of $\bar{R}$.
Then there is an integer $k>0$ such that
$$
\depth (M/I^t M)_\p=\depth (M/I^k M)_\p
$$
for all integers $t\ge k$ and all prime ideals $\p$ of $R$.
\end{thm}

\begin{proof}
The associated graded ring $A=\bigoplus_{i\ge 0}I^i/I^{i+1}$ is a homogeneous ring.
Then $\bigoplus_{i\ge 0}I^i M/I^{i+1} M$ is a finitely generated graded $A$-module.
By Corollary \ref{codepth2} (2) and Lemma \ref{RS4.2}, we find an integer $m>0$ such that 
\begin{equation}\label{last1}
\depth (I^t M/I^{t+1} M)_\p=\depth (I^m M/I^{m+1} M)_\p
\end{equation}
for all integers $t\ge m$ and all prime ideals $\p$ of $R$.
Note that 
\begin{equation}\label{last2}
X:=\supp_R(M) \cap \supp_R(R/I)=\supp_R(M/I^{i} M)
\end{equation}
for any $i>0$.
Applying Corollary \ref{codepth2} (2) to $A=A_0=R$, we see that $U_n^t :=\bigcup_{m\le i\le t} \C_n^{R}(M/I^i M)$ is open for any $t\ge m$ and any $n\ge 0$.
Lemma \ref{RS4.0} implies that there is an integer $l\ge m$ such that $U_n^t=U_n^l$ for all $t\ge l$ and all $n\ge0$.
Put $k=l+1$.
By (\ref{last2}), we have only to show that the following claim holds.
\begin{spacing}{1.2}
\ \textbf{Claim.} \ $\C_n^{R}(M/I^t M)=\C_n^{R}(M/I^k M)$ for all $t\ge k$ and all $n\ge0$.
\end{spacing}
Fix an integer $n\ge 0$.
Let $\p$ be a prime ideal of $R$ belonging to $\C_n^{R}(M/I^t M)$ for some $t\ge k$.
We prove that $\p$ is in $\C_n^{R}(M/I^i M)$ for all $i\ge k$.
We may assume that $\p$ is in $X$.
By (\ref{last1}) and (\ref{last2}), we obtain $r:=\depth (I^k M/I^{k+1} M)_\p=\depth (I^i M/I^{i+1} M)_\p$ and $d:=\dim(M/I^k M)_\p=\dim(M/I^i M)_\p$ for all $i\ge m$.
The prime ideal $\p$ belongs to $\C_n^{R}(M/I^s M)$ for some $m\le s \le l$ since $U_n^t=U_n^l$, which means $\depth (M/I^s M)_\p\ge d-n$.
For each integer $i\ge s$, there is an exact sequence 
$$
0 \to (I^i M/I^{i+1} M)_\p \to (M/I^{i+1} M)_\p \to (M/I^i M)_\p \to 0.
$$
Suppose $r<d-n$.
It follows from \cite[Proposition 1.2.9]{BH} and by induction on $i$ that $\depth (M/I^{i+1} M)_\p=r$ for any $i\ge s$.
In particular, we have $\depth (M/I^t M)_\p=r<d-n$.
This is a contradiction.
Hence, we get $r\ge d-n$.
Similarly, we see by induction on $i$ that $\depth (M/I^i M)_\p\ge d-n$ for any $i\ge s$.
This means $\p$ belongs to $\C_n^{R}(M/I^i M)$ for all integers $i\ge s$.
The proof of claim is now completed.
\end{proof}

We recall a few definitions of notions used in our next result.

\begin{dfn}\label{def rings}
A ring $R$ is said to be \textit{quasi-excellent} if the following two conditions are satisfied.
\begin{enumerate}[\rm(1)]
\item  For all finitely generated $R$-algebras $S$, $\reg(S) = \{\p\in\spec (S)\mid$ the local ring $S_\p$ is regular$ \}$ is open.
\item All the formal fibers of $R_\p$ are regular for all prime ideals $\p$ of $R$.
\end{enumerate}
A ring $R$ is said to be \textit{excellent} if it is quasi-excellent and universally catenary.
A ring in which ``regular'' is replaced with ``Gorenstein'' in the definition of an excellent ring is called an \textit{acceptable ring} \cite{S}.
\end{dfn}

Applying the above theorem, we can prove the main result of this paper.

\begin{cor}\label{main cor}
Let $R$ be a ring and $I$ an ideal of $R$.
Let $M$ be a finitely generated $R$-module.
Put $\bar{R}=R/(I+\operatorname{Ann}_R (M))$.
Then there is an integer $k>0$ such that
$$
\depth (M/I^t M)_\p=\depth (M/I^k M)_\p
$$
for all integers $t\ge k$ and all prime ideals $\p$ of $R$ in each of the following cases.\\
{\rm(1)} $M$ is Cohen--Macaulay. 
{\rm(2)} $M/I^n M$ is Cohen--Macaulay for some $n>0$.
{\rm(3)} $\bar{R}$ is a homomorphic image of a Cohen--Macaulay ring. 
{\rm(4)} $\bar{R}$ is semi-local.
{\rm(5)} $\bar{R}$ is excellent.
{\rm(6)} $\bar{R}$ is quasi-excellent and catenary.
{\rm(7)} $\bar{R}$ is acceptable.
\end{cor}

\begin{proof}
In any of the latter three cases, the assertion follows from Theorem \ref{local depth}.

(1): It is seen by \cite[Theorem 2.1.3 (b)]{BH} and \cite[Theorem 5.4]{K} that $R/\operatorname{Ann}_R (M)$ is catenary and $\cm(R/\p)$ contains a nonempty open subset of $\spec (R/\p)$ for any $\p\in\supp_{R}(M)$.
Thus the assertion follows from Theorem \ref{local depth}.

(2) and (3): The assertion can be shown in a similar way as in the proof of (1); see \cite[Corollary 5.6]{K}.

(4): We may assume that $R$ is local.
Let $\hat{R}$ be the completion of $R$ and $\hat{M}$ the completion of $M$.
For any prime ideal $\p$ of $R$, there exists a prime ideal $\q$ of $\hat{R}$ such that $\p=\q\cap R$ because $\hat{R}$ is faithfully flat over $R$.
It follows from \cite[Proposition 1.2.16 (a)]{BH} that
$$
\depth_{R_\p} (M/I^t M)_\p=\depth_{\hat{R}_\q} (\hat{M}/I^t \hat{M})_\q-
\depth_{\hat{R}_\q} (\hat{R}_\q/\p \hat{R}_\q)
$$
for any $t>0$.
The assertion follows from (3) since $\hat{R}$ is a homomorphic image of a regular local ring.
\end{proof}

The assumptions about the ring $\bar{R}$ in Theorem \ref {local depth} and Corollary \ref{main cor} are satisfied if so does the ring $R$.
The above corollary recovers the theorem of Rotthaus and \c{S}ega \cite{RS}.

\begin{cor}[Rotthaus--\c{S}ega]\label{recover1}
Let $R$ be an excellent ring and let $M$ be a Cohen--Macaulay $R$-module.
Let $I$ be an ideal of $R$ which is not contained in any minimal prime ideal of $M$.
Then there is an integer $k>0$ such that $\depth (M/I^t M)_\p=\depth (M/I^k M)_\p
$ for all integers $t\ge k$ and all prime ideals $\p$ of $R$.
\end{cor}

By using the technique of the proof of Theorem \ref{local depth}, the depths of localizations of $M/I^{n+1}M$ can be measured by those of $I^nM/I^{n+1}M$ for each integer $n$.
We provide two examples where Corollary \ref{main cor} is applicable, but Corollary \ref{recover1} is not.

\begin{ex}
Let $R=K \llbracket x, y, z, w \rrbracket/(xy-zw)$ be a quotient of a formal power series
ring over a field $K$.
Take the ideal $I=(x)$ of $R$ and the finitely generated $R$-module $M=R/(w)$.
The ring $R$ is a local hypersurface of dimension 3 that has an isolated singularity.
The module $M$ is Cohen--Macaulay, and all elements of $I$ are zero-divisors of $M$.
Then $M$ is also a module over $A=K \llbracket x, y, z \rrbracket$.
We see that $M\simeq A/(xy)$, $\ M/I^n M\simeq A/(x^n, xy)$ and $\ I^nM/I^{n+1}M\simeq A/(x,y)$.
Let $\p$ be a prime ideal of $A$.
A similar argument to the latter part of the proof of Theorem \ref{local depth} shows that $\depth (M/I^n M)_\p=\height\p - 2$ for any integer $n\ge2$ if $\p$ contains the ideal $(x,y)$ of $A$;
otherwise, we have $(M/I^{n+1} M)_\p \simeq (M/I^nM)_\p$ for any integer $n\ge1$.
This says that the integer $k=2$ satisfies the assertion of Corollary \ref{main cor}.
\end{ex}

\begin{ex}
Let $R=K [x, y, z]$ be a polynomial ring over a field $K$.
Take the ideal $I=(x)$ of $R$ and the finitely generated $R$-module $M=R/(x^m y, x^m z)$, where $m>0$.
The ring $R$ is regular but not local.
All elements of $I$ are zero-divisors of $M$.
The $R$-module $M$ is not Cohen--Macaulay; see \cite[Theorem 2.1.2 (a)]{BH}.
We have
$$
M/I^n M\simeq R/(x^n, x^m y, x^m z), \quad
I^nM/I^{n+1}M\simeq 
\begin{cases}
        {R/(x) \quad (n<m)}\\
        {R/(x, y, z) \quad (n\ge m)}.
\end{cases}
$$
Let $\p$ be a prime ideal of $R$.
Suppose $\p=(x,y,z)$.
We get $\depth (M/I^n M)_\p=2$ for any $0\le n\le m$.
On the other hand, we obtain $\depth (M/I^n M)_\p=0$ for any $n>m$ since the submodule $I^{n-1}M/I^n M$ of $M/I^n M$ is isomorphic to $R/\p$.
It is seen that $(M/I^{n+1} M)_\p \simeq (M/I^nM)_\p$ for any integer $n\ge m$ if $\p\ne (x,y,z)$.
This says that the integer $k=m+1$ satisfies the assertion of Corollary \ref{main cor}.
\end{ex}

For modules all of whose localizations have the same depth, the notion of a regular sequence is consistent.

\begin{prop}\label{grade}
Let $R$ be a ring.
Let $M$ and $N$ be finitely generated $R$-modules.
Suppose that $\depth (M_\p)=\depth (N_\p)$ for all prime ideals $\p$ of $R$.
Then, for any sequence $\bm{x}=x_1,\ldots,x_n$ in $R$, $\bm{x}$ is an $M$-regular sequence if and only if it is an $N$-regular sequence.
In particular, $\grade(J, M)=\grade(J, N)$ for any ideal $J$ of $R$.
\end{prop}

\begin{proof}
We observe that $\supp_R(M)=\supp_R(N)$.
We prove the proposition by induction on $n$.
It is seen that $\ass_R(M)=\ass_R(N)$ and $\supp_R(M/x M)=\supp_R(N/x N)$ for any $x\in R$ by assumption.
This says that the assertion of the proposition holds in the case $n=1$.
Suppose $n>1$.
We may assume that $\bm{x}'=x_1,\ldots,x_{n-1}$ is a regular sequence on both $M$ and $N$.
Then we see that $\depth (M/\bm{x}'M)_\p=\depth (N/\bm{x}'N)_\p$ for all prime ideals $\p$ of $R$.
Applying the case $n=1$ shows the assertion.
\end{proof}

The following two results are direct corollaries of Corollary \ref{main cor} and Proposition \ref{grade}.
The latter corollary recovers the theorem of Kodiyalam \cite{Ko}.
Note that, unlike Corollary \ref{recover2}, the integer $k$ does not depend on the ideal $J$ in Corollary \ref{sequence cor}.

\begin{cor}\label{sequence cor}
Let $R$ be a ring, $I$ an ideal of $R$, and $M$ a finitely generated $R$-module.
Suppose that we are in one of the cases of {\rm Corollary \ref{main cor}}.
Then there is $k>0$ such that $\bm{x}$ is an $M/I^t M$-regular sequence if and only if it is an $M/I^k M$-regular sequence for all integers $t\ge k$ and all sequences $\bm{x}=x_1,\ldots,x_n$ in $R$.
In particular, $\grade(J, M/I^t M)=\grade(J, M/I^k M)$ for all integers $t\ge k$ and all ideals $J$ of $R$.
\end{cor}

\begin{cor}[Kodiyalam]\label{recover2}
Let $R$ be a local ring, $I$ and $J$ ideals of $R$, and $M$ a finitely generated $R$-module.
Then there is $k>0$ such that $\grade(J, M/I^t M)=\grade(J, M/I^k M)$ for all integers $t\ge k$.
\end{cor}

Finally, we remark that the theorem proved by Brodmann \cite{B} is recovered from this corollary.

\begin{rmk}\label{ASS}
Let $R$ be a ring, $I$ an ideal of $R$, and $M$ a finitely generated $R$-module.
Lemma \ref{RS1.3} asserts that $\bigcup_{i\ge 0} \ass_{R}(I^i M/I^{i+1} M)$ is a finite set.
By induction on $n>0$, it is seen that $\ass_{R}(M/I^{n} M)$ is contained in 
$\bigcup_{i=0}^{n-1} \ass_{R}(I^i M/I^{i+1} M)$; see \cite[Theorem 6.3]{Mat}.
The set $X:=\bigcup_{n>0} \ass_{R}(M/I^{n} M)$ is also a finite set.
It follows from Corollary \ref{recover2} that there is an integer $k>0$ such that 
$$
\depth (M/I^t M)_\p=\depth (M/I^k M)_\p
$$
for all integers $t\ge k$ and all prime ideals $\p$ of $R$ belonging to $X$.
This says that for all integers $t\ge k$, 
$$
\ass_{R}(M/I^{t} M)=\ass_{R}(M/I^{k} M).
$$
\end{rmk}

\begin{ac}
The author would like to thank his supervisor Ryo Takahashi for valuable comments.
\end{ac}

\end{document}